\documentclass[11pt]{amsart}
\usepackage{latexsym}
\usepackage{amssymb}
\usepackage{amsmath}

\usepackage[pdftex,colorlinks]{hyperref}

\usepackage{amsfonts}

\newtheorem{theorem}{Theorem}[section]
\newtheorem{lemma}[theorem]{Lemma}
\newtheorem{proposition}[theorem]{Proposition}
\newtheorem{corollary}[theorem]{Corollary}

\newtheorem{remark}[theorem]{Remark}

\newtheorem{question}[theorem]{Question}

\newcommand\sspp{\mathop{\rm span}}
\newcommand\vn{\mathop{\rm VN}}
\newcommand\supp{\mathop{\rm supp}}

\newcommand\Bim{\mathop{\rm Bim}}
\newcommand\Sat{\mathop{\rm Sat}}

\newcommand\nul{\mathop{\rm null}}
\newcommand\esssup{\mathop{\rm esssup}}

\newcommand\an{^{-1}}

\newcommand\cb{\mathop{\rm cb}}

\def\gl{\lambda}

\newcommand{\cl}[1]{\mathcal{#1}}
\newcommand{\bb}[1]{\mathbb{#1}}
\newcommand{\du}[2]{\left\langle{#1},{#2} \right\rangle} %\du{T}{u}
\def\ot{\otimes}
\newcommand{\nor}[1]{\left\Vert #1\right\Vert}    %\nor{x^2}
\newcommand{\sca}[1]{\left(#1\right)} %

\newcommand{\red}[1]{\textcolor{red}{#1}}

\begin{document}

\title[Ideals and bimodules]{Ideals of $A(G)$ and bimodules
over maximal abelian selfadjoint algebras}

\author{M. Anoussis, A. Katavolos and  I. G. Todorov}

\dedicatory{To the memory of Bill Arveson, with gratitude}

\address{Department of Mathematics, University of the Aegean,
Samos 83 200, Greece}

\email{mano@aegean.gr}

\address{Department of Mathematics, University of Athens,
Athens 157 84, Greece}

\email{akatavol@math.uoa.gr}

\address{Pure Mathematics Research Centre, Queen's University Belfast,
Belfast BT7 1NN, United Kingdom}

\email{i.todorov@qub.ac.uk}


\begin{abstract} This paper is concerned with
weak* closed masa-bimodules generated by $A(G)$-invariant subspaces of
$\vn(G)$. An annihilator formula is established, which is used to characterise the
weak* closed subspaces of $\cl B(L^2(G))$ which are
invariant under both Schur multipliers and
a canonical action of $M(G)$ on $\cl B(L^2(G))$ via completely bounded maps. 
We study the special cases of extremal ideals with a given null set
and, for a large class of groups, we establish a link between
relative spectral synthesis and relative operator synthesis.
\end{abstract}

\maketitle

\section{Introduction}\label{s_intro}

Let $G$ be a locally compact group. The algebra $M^{\cb}A(G)$ of
completely bounded multipliers of the Fourier algebra $A(G)$, introduced
in \cite{deCH}, has played a pivotal role in both Harmonic Analysis and
Operator Algebra Theory.
It was shown by J. E. Gilbert and by M. Bo\.{z}ejko and G. Fendler
in \cite{bf} (see also \cite{j} and \cite{spronk}) that
the map $N$ which sends a function $f : G\to \bb{C}$ to the function
$Nf : G\times G\to \bb{C}$ given by
$Nf(s,t) = f(ts^{-1})$,
carries $M^{\cb}A(G)$ isometrically into the
algebra of Schur multipliers $\frak{S}(G)$ on $G\times G$.
This result has led to 
fruitful interaction between
the two areas, see {\it e.g.} \cite{pi}, \cite{nrs} and \cite{spronk}.

The weak* closed subspaces of the von Neumann algebra $\vn(G)$
that are invariant under $A(G)$ are precisely the annihilators
of (closed) ideals $J\subseteq A(G)$.
On the other hand, the weak* closed subspaces of
the space $\cl B(L^2(G))$ of 
bounded operators on $L^2(G)$ which are
invariant under all Schur multipliers are precisely the
(weak* closed) masa-bimodules in $\cl B(L^2(G))$,
that is, invariant under the map $T\to M_fTM_g$ where 
$f,g\in L^\infty(G)$ (or, under left and right compostion with multiplication
operators from $L^\infty(G)$).  

Thus, given a closed ideal $J\subseteq A(G)$, there are two
natural ways to construct a weak* closed masa-bimodule
in $\cl B(L^2(G))$: (a) one may first consider the
norm closed masa-bimodule $\Sat(J)$ of $T(G)$ suitably generated by
$N(J)$ and then take the annihilator of $\Sat(J)$ in $\cl B(L^2(G))$, or
(b) one may first take the annihilator $J^{\perp}$ of $J$ in $\vn(G)$
and then generate a weak* closed masa-bimodule $\Bim(J^{\perp})$.
One of our main results, Theorem \ref{th_satlcg},
is that these two operations have the same
outcome; in other words, the diagram
\[
\begin{matrix}
J & \stackrel{\perp}{\longrightarrow} & J^{\perp}\\
\downarrow & &  \downarrow\\
\Sat(J) & \stackrel{\perp}{\longrightarrow} & \Bim(J^{\perp})
\end{matrix}
\]

\medskip

\noindent is commutative.
The proof uses the techniques developed by J. Ludwig, N. Spronk
and L. Turowska in \cite{spronk_tur} and \cite{lt}.
Some of the results in Section \ref{s_invmabi}
also appear in the aforementioned papers;
we have chosen to present complete arguments in order
to clarify some details.

Using this result, we present  a unified approach to some problems of
Harmonic Analysis on $G$.
In particular, in Section \ref{s_eb}, we look at the special cases where $J$
is the minimal, or the maximal, ideal of $A(G)$ with a given null set $E\subseteq G$.
The main result here is Theorem \ref{th_mma}; as a corollary, we obtain
the result established in \cite{lt} that if $A(G)$ possesses an approximate
identity then a closed set $E\subseteq G$ satisfies spectral synthesis
if and only if the set $E^* = \{(s,t) : ts^{-1}\in E\}$ satisfies
operator synthesis.

The connection between spectral synthesis and operator synthesis
was discovered by
W. B. Arveson in \cite{a}.  The above result is due to J. Froelich \cite{f}  for $G$ abelian
and  to N. Spronk and L. Turowska \cite{spronk_tur} for $G$ compact.
J. Ludwig and L. Turowska \cite{lt} show that
 a closed subset $E$ of a locally compact group $G$
satisfies {\em local } spectral synthesis
if and only if $E^*$  satisfies
operator synthesis; local spectral synthesis coincides with spectral synthesis
when $A(G)$ has an approximate identity.

Spectral synthesis relative to a fixed  $A(G)$-invariant subspace of $\vn(G)$ 
  was introduced for locally compact groups by E. Kaniuth and A.T. Lau in
\cite{kl}. In \cite{pp}, the authors define relative operator synthesis
for subsets of $G\times G$, where $G$ is compact,
and link it to relative spectral synthesis.  In Section \ref{s_rs},
using our results, and  assuming that 
the   $A(G)$-invariant subspace of $\vn(G)$  is weak* closed we prove an analogous relation    for locally compact groups
for which $A(G)$ possesses
an approximate identity.
We note that this class contains, but is larger than, the class of amenable groups.

As another application, we are able   to identify the weak* closed subspaces
of $\cl B(L^2(G))$ that are invariant
under both Schur multiplication and an action of the measure algebra $M(G)$.
More precisely, let
$\Gamma : M(G)\to \cl B(\cl B(L^2(G)))$ be the representation of $M(G)$ given by
$$\Gamma(\mu) (T) = \int_{G} \rho_rT\rho_r^*d\mu(r), \ \ \ T\in \cl B(L^2(G)).$$
This action was studied
by F. Ghahramani, M. Neufang, Zh.-J. Ruan, R. Smith, N. Spronk and E. St\o{}rmer
in \cite{gh}, \cite{n}, \cite{nrs}, \cite{smith_spronk}, \cite{stor}, among others.

The maps $\Gamma(\mu)$ are precisely those
weak* continuous completely bounded maps
on $\cl B(L^2(G))$ that are $\vn(G)$-bimodule maps and leave the
multiplication masa invariant \cite{n}, \cite{nrs}.
In Section \ref{s_invsub}, we show that the set $\cl L$ of all weak* closed
subspaces of $\cl B(L^2(G))$ that are invariant under both $\frak{S}(G)$ and
$\Gamma(M(G))$ consists precisely of the masa-bimodules of the form
 $\Bim(J^{\perp})$, where $J\subseteq A(G)$ is a closed ideal;
we also determine the lattice structure of $\cl L$.

In the presence of an approximate identity in $A(G)$, we show that
the generating invariant subspace of a bimodule of the form $\Bim(\cl X)$ can be
recovered by taking the intersection with $\vn(G)$.
Thus, the map $\cl X\rightarrow \Bim (\cl X)$ from the class of
weak* closed invariant subspaces of $\vn(G)$ 
to the class of weak* closed masa bimodules  
of $B(L^2(G))$ is in this case one-to-one.

\section{Preliminaries}\label{s_prel}

If $(X,m)$ is a $\sigma$-finite measure space,
we write $L^p(X)$   for $L^p(X,m)$.
For $\phi\in L^{\infty}(X)$, let $M_{\phi}$ be the operator on $L^2(X)$
of multiplication by $\phi$.
The collection
$\cl D_X = \{M_{\phi} : \phi\in L^{\infty}(X)\}$ is a maximal
abelian selfadjoint algebra (masa, for short). 

Let $X$ and $Y$ be  standard Borel spaces
(that is, Borel isomorphic to Borel subsets of
complete separable metric spaces), equipped with $\sigma$-finite 
measures $m$ and $n$.
A subset $E\subseteq X\times Y$ is called {\it marginally null} if
$E\subseteq (X_0\times Y)\cup(X\times Y_0)$, where
$m(X_0) = n(Y_0) = 0$; we write $E \simeq \emptyset$.
Two functions $h_1,h_2 : X\times Y\to \bb{C}$ are said to be equal
{\em marginally almost everywhere (m.a.e.)} or
{\em marginally equivalent} if the set
$\{(x,y) : h_1(x,y)\neq h_2(x,y)\}$ is marginally null.

Let $T(X,Y)$
be the projective tensor product $L^2(X)\hat{\otimes} L^2(Y)$.
Every element $h\in T(X,Y)$ is
an absolutely convergent series
\[
 h=\sum_{i=1}^{\infty} f_i\ot g_i, \quad f_i\in L^2(X), g_i\in L^2(Y), i\in \bb{N},
\]
where
$\sum_{i=1}^{\infty}\nor{f_i}_2^2<\infty$ and $\sum_{i=1}^{\infty}\nor{g_i}_2^2<\infty$.
Such an element $h$
may be considered either as a function $h:X\times Y\to\bb C$, defined
marginally almost everywhere and given by
\[
 h(x,y)=\sum_{i=1}^{\infty} f_i(x )g_i(y),
\]
or as an element of the predual of the space $\cl B(L^2(X),L^2(Y))$
of all bounded linear operators from $L^2(X)$ into $L^2(Y)$ via the pairing
\[
 \du{T}{h}_t := \sum_{i=1}^{\infty} \sca{Tf_i,\bar g_i}.
\]
We denote by $\nor{h}_t$ the norm of $h\in T(X,Y)$
and note that if  
$\phi\in L^{\infty}(X)$ and $\psi\in L^{\infty}(Y)$,
then the function $(\phi\otimes \psi) h$ belongs to $T(X,Y)$; thus,
$T(G)$ has a natural  $(L^{\infty}(X),L^{\infty}(Y))$-module structure.

Let $\frak{S}(X,Y)$ be the multiplier algebra of $T(X,Y)$; by definition, a
measurable function $w : X\times Y\rightarrow \bb{C}$ belongs to $\frak{S}(X,Y)$ if
the map $m_w: h\to wh$ leaves $T(X,Y)$ invariant, that is, if
$wh$ is marginally equivalent to a function from $T(X,Y)$,
for every $h\in T(X,Y)$.
The elements of $\frak{S}(X,Y)$ are called \emph{(measurable) Schur multipliers}.
The closed graph theorem can be used to show that
$m_w$ is automatically
a bounded operator; hence it has a dual
\[
S_w : \cl B(L^2(X),L^2(Y)) \rightarrow \cl B(L^2(X),L^2(Y)),
\]
given by
\[\du{S_w(T)}{h}_t = \du{T}{wh}_t, \ \ \ h\in T(X,Y), \; T\in \cl B(L^2(X),L^2(Y)).\]
If $\phi\in L^{\infty}(X)$ and $\psi\in L^{\infty}(Y)$,
one finds that $S_{\phi\otimes\psi}(T)=M_\psi TM_\phi$, $T\in \cl B(L^2(X),L^2(Y))$.
It follows that if
$k\in L^2(Y\times X)$, $T_k$
is the Hilbert-Schmidt operator from $L^2(X)$ into $L^2(Y)$ given by
$(T_k f)(y) = \int_X k(y,x) f(x) dm(x)$ ($f\in L^2(X)$) and
$w \in \frak{S}(X,Y)$,
then $S_w(T_k) = T_{w^\flat k}$ where
$w^\flat : Y\times X\to \bb{C}$ is the function 
$w^\flat (x,y)=w(y,x)$.

It can be shown  (\cite{peller_two_dim}, see also \red{\cite{haag, kp}} and
 \cite{spronk}) 
that $w\in \frak{S}(X,Y)$ if and only if $w$ can be represented in the form
\[w(x,y) = \sum_{k=1}^{\infty} a_k(x) b_k(y), \ \ \ \
\mbox{for almost all } (x,y)\in X\times Y,\]
where  $(a_k)_{k\in \bb{N}}\subseteq L^{\infty}(X)$ and
$(b_k)_{k\in \bb{N}}\subseteq L^{\infty}(Y)$ are sequences of functions with
$\esssup_{x\in X} \sum_{k=1}^{\infty} |a_k(x)|^2 < \infty$ and
$\esssup_{y\in Y} \sum_{k=1}^{\infty} |b_k(y)|^2 < \infty$.
In this case, 
\[
S_w(T)=\sum_{k=1}^\infty M_{b_k}TM_{a_k}, \ \ \ T\in\cl   B(L^2(X),L^2(Y)),
\]
where, for every $T\in\cl   B(L^2(X),L^2(Y))$, the series converges in the weak* topology. 
We write $w=\sum\limits_{k=1}^\infty a_k\ot b_k$
as a formal series. Moreover, the norm of $S_w$ as an operator on  
$\cl B(L^2(X),L^2(Y))$ (which of course also equals the norm of the operator
$m_w$ on $T(X,Y)$) is given by
\begin{align*}
\nor{S_w}
&=\inf\left\{  \nor{\sum_{k=1}^{\infty} |a_k|^2}_\infty^{1/2}
\nor{\sum_{k=1}^{\infty} |b_k|^2}_\infty^{1/2}:
\text{ all rep's }\; w=\sum_{k=1}^{\infty} a_{k}\otimes b_{k}\right\} \\
\end{align*}
We denote this quantity by $\nor{w}_{\frak{S}}$.
Note that if $w = \sum\limits_{k=1}^\infty a_k\ot b_k \in \frak{S}(X,Y)$ then
$w^\flat=\sum_{k=1}^{\infty} b_k\ot a_k \in \frak{S}(Y,X)$ and
$\nor{w}_{\frak{S}} = \nor{w^\flat}_{\frak{S}}$.

If $w\in \frak{S}(X,Y)$, the operator $S_w$ is a $(\cl D_Y,\cl D_X)$-module map, while
the operator $m_w$ is a $(L^\infty(X),L^{\infty}(Y))$-module map.
In fact, a weak* closed subspace $\cl U\subseteq \cl B(L^2(X),L^2(Y))$
is a \emph{masa-bimodule} in the sense that
$BTA\in \cl U$ for all $A\in \cl D_X$, $B\in \cl D_Y$ and
$T\in \cl U$, if and only if $\cl U$ is
invariant under the mappings $S_w$, $w\in \frak{S}(X,Y)$ 
\cite[Proposition 3.2]{eletod}.
It follows by duality that a norm closed subspace $V\subseteq T(X,Y)$ is
an $(L^{\infty}(X),L^{\infty}(Y))$-module if and only if
it is invariant under the mappings $m_w$, $w\in \frak{S}(X,Y)$.

\medskip

Throughout, $G$ will denote a second countable
locally compact group. 

We now summarise some results from non-commutative
harmonic analysis. 
All spaces $L^p(G)$ are with respect to 
left Haar measure  $m$;  $dm(x)$ is shortened to $dx$ and the modular function
is denoted by $\Delta$.
If $A, B\subseteq G$ we write $A^{-1} = \{x^{-1} :
x\in A\}$ and $AB = \{xy : x\in A, y \in B\}$.
Denote by $\lambda : G\rightarrow \cl B(L^2(G))$, $s\to \lambda_s$, the left
regular representation  and write
$(f,g)$ for the inner product of the elements $f,g\in L^2(G)$.
We set $T(G) = T(G,G)$,  $\mathcal{B}(L^2(G)) = \mathcal{B}(L^2(G), L^2(G))$ and
$\frak{S}(G) = \frak{S}(G,G)$.
The \emph{group von Neumann algebra} of $G$ is the algebra
$$\vn(G) = \sspp\{\lambda_x : x\in G\}^{-w*},$$ acting on $L^2(G)$,
while the \emph{Fourier algebra} $A(G)$ of $G$ \cite{eym}
is the (commutative, regular, semi-simple) Banach algebra
consisting of
all complex functions $u$ on $G$ of the form
\begin{equation}\label{ag}
u(x) = (\lambda_x \xi,\eta),\ \  x\in G, \mbox{ where } \xi,\eta\in L^2(G).
\end{equation}
Multiplication in $A(G)$ is pointwise, while
the norm $\|u\|$ of an element $u\in A(G)$
is the infimum of the products $\|\xi\|_2\|\eta\|_2$ over all
representations (\ref{ag}) of $u$. The spectrum of $A(G)$
is identified with $G$ via point evaluations.

Every element $\tau$ of the dual $A(G)^*$ 
defines a bounded operator
$T_\tau$ on $L^2(G)$ by the formula
\[ \du{\tau}{u}_a :=(T_\tau\xi,\eta)\]
(the symbol $\du{\cdot}{\cdot}_a$ is used to denote the duality between $A(G)$
and $A(G)^*$), where $u\in A(G)$ is given by (\ref{ag}). 
The map 
\[\tau\to T_\tau : A(G)^*\to\cl B(L^2(G))\]
sends $A(G)^*$ isometrically and weak* homeomorphically  onto $\vn(G)$.

Note that the spaces $A(G)^*$ and $\vn(G)$ are usually identified in the literature
and the map
$\tau\to T_{\tau}$ is suppressed;  we have chosen to retain it in order
to emphasize the different dualities  used in this paper.
The algebra $\vn(G)$ is a Banach $A(G)$-module
under the operation $$(u,T_{\tau})\to u T_\tau = T_{\tau'},$$ where $\tau'$ is 
defined by the relation
$$\du{\tau'}{v}_a=\du{\tau}{uv}_a, \ \ \ v\in A(G).$$

The predual $P: T(G)\to A(G)$
of the map $\tau\to T_{\tau}:A(G)^*\to\cl B(L^2(G))$ is the contraction 
given by
\begin{equation}\label{P}
\du{\tau}{P(h)}_a = \du{T_\tau}{h}_t, \ \ \ \tau\in A(G)^*, \ h\in T(G).
\end{equation}

To obtain an explicit formula for $P$, take 
$\tau$ such that $T_\tau=\gl_s$ and recall that
$\du{\tau}{u}_a=u(s)$, $s\in G$.
If $h\in T(G)$ is of the form $h=\sum_{i=1}^{\infty} f_i\ot g_i$,
where 
$\sum_{i=1}^{\infty}\|f_i\|_2^2 < \infty$ and $\sum_{i=1}^{\infty}\|g_i\|_2^2 < \infty$,
then
\begin{align*}
P(h)(s) &= \du{\tau}{P(h)}_a=\du{\gl_s}{h}_t
=\sum_{i=1}^{\infty} \sca{\lambda_s f_i,\bar g_i} \\
&= \sum_{i=1}^{\infty}\int_G f_i(s^{-1}t)g_i(t) dt
= \int_G \sum_{i=1}^{\infty} f_i(s^{-1}t)g_i(t) dt
\end{align*}
by Proposition \ref{rem2} below.
Thus
\begin{equation}\label{eq_p}
P(h)(s) = \int_G h(s\an t,t)dt,  \ \ \ s\in G.
\end{equation}

The space $A(G)$ has a canonical
operator space structure arising from its identification with the predual of $\vn(G)$
(the reader is referred to \cite{er}, \cite{pa}, \cite{pi} for the basic notions of
operator space theory). We write
\[MA(G) = \{v : G\rightarrow \bb{C} \ : \ vu\in A(G) \mbox{ for all } u\in A(G)\}\]
for the multiplier algebra of $A(G)$; 
the set of all $v\in MA(G)$ for which the map
$u\rightarrow vu$ on $A(G)$
is completely bounded will be denoted $M^{\cb}A(G)$ and equipped with the 
completely bounded norm.

Define \[N : L^{\infty}(G)\rightarrow L^{\infty}(G\times G)
\quad\text{by }\quad N(f)(x,y) = f(yx^{-1}).\]
We warn the reader that our definition of the map $N$ differs from
the one used in \cite{spronk_tur, lt}, where the expression
$f(xy^{-1})$ is used instead of $f(yx^{-1})$.

The following result \cite{bf}, \cite{spronk} (see also \cite{j}) will be used in the sequel.

\begin{theorem}\label{spro}
The map $u\rightarrow N(u)$ is an isometry 
from $M^{\cb}A(G)$ 
into $ \frak{S}(G)$.
Moreover, $N(M^{\cb}A(G))$ equals the space of those 
$w\in \frak{S}(G)$ for which $w(sr,tr) = w(s,t)$ for every $r\in G$ and
marginally almost all $s,t$.
\end{theorem}

If $G$ is compact then $T(G)$ contains the constant functions, and
Theorem \ref{spro} implies that $N$ takes values in $T(G)$.

If $u\in A(G)$, $h\in T(G)$ and $t\in G$ then, using (\ref{eq_p}), we have
\begin{align}
    P(N(u)h)(t) &= \int N(u)(t\an s,s)h(t\an s,s)ds  \nonumber \\
&=\int u(s(t\an s)\an)h(t\an s,s)ds 
= u(t)P(h)(t) \nonumber \\
\text{so }\quad   P(N(u)h) &= uP(h). \label{pn}
  \end{align}

\section{Ideals and bimodules}\label{s_invmabi}

The main result of this section is the annihilator formula
of Theorem \ref{th_satlcg}.
We start by explaining its main ingredients.
Given a closed ideal $J$ of $A(G)$, we will abuse notation and identify
its annihilator $J^\perp$ with a (weak* closed) subspace of $\vn(G)$.
The space $J^{\perp}$ is \emph{invariant}, that is,
it is an $A(G)$-submodule of $\vn(G)$;
it is easy to see that every weak* closed invariant subspace of 
$\vn(G)$  arises in this way.
Similarly, there is a bijective correspondence between the class 
of all norm closed $L^{\infty}(G)$-bimodules in $T(G)$
and the class of all weak* closed masa-bimodules in $\cl B(L^2(G))$, given
by taking annihilators and pre-annihilators.

Given any  weak* closed invariant subspace $\cl X$ of $\vn(G)$,
we let $\Bim(\cl X)\subseteq \cl B(L^2(G))$
be the weak* closed masa-bimodule generated by $\cl X$.
It is not hard to see 
 that
\[
\mathrm{Bim}(\mathcal{X})=\overline{[\mathfrak{S}(G)\mathcal{X}]}^{w^*}.
\]

Note that if $\cl U\subseteq \cl B(L^2(G))$
is a weak* closed masa-bimodule
then $\cl U \cap \vn(G)$ is a
weak* closed invariant subspace of $\vn(G)$;
indeed, if $u\in A(G)$ and $T \in \cl U \cap \vn(G)$ then
$uT = S_{N(u)}(T)$. 

Given a closed ideal $J\subseteq A(G)$, we wish to define, similarly, a
norm closed $L^{\infty}(G)$-bimodule in $T(G)$ \lq\lq generated by'' $J$.
To this end, suppose first that $G$ is compact. Then, as pointed out in
Section \ref{s_prel}, $N(J)\subseteq T(G)$. Hence, one may consider
the norm closed $L^{\infty}(G)$-bimodule of $T(G)$
generated by $N(J)$, that is, the space
$\overline{[\frak{S}(G) N(J)]}^{\|\cdot\|_t}$.
If $G$ is not compact, the map $N$ does not take values in $T(G)$ but in
$\frak{S}(G)$. 
However, if $u\in A(G)$ then $N(u)$ belongs to $T(G)$ 'locally' in the sense that 
$N(u)\chi_{L\times L}\in T(G)$ for every compact subset $L\subseteq G$
(indeed, for such $L$, the function $\chi_{L\times L}$ is in $T(G)$ and since 
$N(u)$ is a multiplier of $T(G)$, the claim follows).
Hence we may consider the
closed $L^\infty(G)$-bimodule of $T(G)$ generated by  the set
\[
\{N(u)\chi_{L\times L}: u \in J, L\  \text{compact, } \ L\subseteq G \}.
\]
We will denote  this bimodule by $\Sat(J)$.
This bimodule may also be written as follows

\begin{proposition}\label{newsat}
If  $J\subseteq A(G)$ is a closed ideal, then
\[ \Sat(J)=\overline{[N(J)T(G)]}^{\|\cdot\|_t}. \]
\end{proposition}
\proof It is clear that $\Sat(J)\subseteq \overline{[N(J)T(G)]}^{\|\cdot\|_t}$.
For the converse, consider $u \in J$ and  $h \in T(G)$. It follows from 
Lemma \ref{tgsg} below that there are compact subsets $K_n$ and $L_n$ of $G$
such that $h\chi_{K_n\times L_n} \in \mathfrak{S}(G)$ and $h$ is the
$\|\cdot\|_t$-limit of the sequence $(h\chi_{K_n\times L_n})_{n}$.
Set $M_n=K_n\cup L_n$. Then 
$h\chi_{K_n\times L_n}=h\chi_{K_n\times L_n}\chi_{M_n\times M_n}$
and $N(u)h=\lim_n N(u)h\chi_{K_n\times L_n}\chi_{M_n\times M_n}$. But
$N(u)\chi_{M_n\times M_n} \in \Sat(J)$ and 
$ h\chi_{K_n\times L_n} \in \mathfrak{S}(G)$.  Thus
\[N(u)h=\lim_{n\to \infty} N(u)h\chi_{K_n\times L_n}\chi_{M_n\times M_n}  \in \Sat(J).\] \qed

\medskip

The rest of the section is devoted to the proof of the following theorem.

\begin{theorem}\label{th_satlcg}
Let $J\subseteq A(G)$ be a closed ideal.
Then
$\Sat(J)^{\perp} = \Bim (J^{\perp})$.
\end{theorem}

We need several preliminary results.

\begin{proposition}\label{rem2}
Let $h = \sum_{i=1}^{\infty} f_i \otimes g_i\in T(G)$ and $s,t\in G$.
Then the function $h_{s,t} : G\to \bb{C}$ given by
$h_{s,t}(r) = h(sr,tr)$, $r\in G$, belongs to $L^1(G)$ and
$\nor{h_{s,t}}_1\le\nor{h}_t$.

In particular,
the sequence $(u_n)_{n\in \bb{N}}$ where $u_n(r)=\sum_{i=1}^n f_i(sr)g_i(tr)$,
converges in the norm of $L^1(G)$ and hence,
for all $f\in L^{\infty}(G)$,
\[
\int_G  f(r)h(sr,tr)dr
= \sum_{i=1}^{\infty} \int f(r) f_i(sr) g_i(tr)dr \, .
\]
\end{proposition}
\begin{proof}
The argument below is due to Ludwig - Turowska
\cite[Proof of Theorem 4.11]{lt}. We reproduce it for completeness:
For
each $s,t\in G$, applying the Cauchy-Schwartz inequality, we obtain
\begin{eqnarray*}
\int_G|h(sr,tr)|dr&\leq&\int_G\sum_{i=1}^{\infty}|f_i(sr) g_i(tr)|dr\\
\nonumber&\leq&\int_G\left(\sum_{i=1}^{\infty}|f_i(sr)|^2\right)^{1/2}
\left(\sum_{i=1}^{\infty}|g_i(tr)|^2\right)^{1/2}dr\\
\nonumber&\leq&\left(\sum_{i=1}^{\infty}\int_G|f_i(sr)|^2dr\right)^{1/2}
\left(\sum_{i=1}^{\infty}\int_G|g_i(tr)|^2dr\right)^{1/2}\\
\nonumber&=&
\left(\sum_{i=1}^{\infty}\|f_i\|_2^2\right)^{1/2}
\left(\sum_{i=1}^{\infty}\|g_i\|_2^2\right)^{1/2}<\infty.
\end{eqnarray*}
Taking the infimum over all representations of $h$,
we obtain
\begin{equation}\label{l1estim}
\nor{h_{s,t}}_1\le\nor{h}_t.
\end{equation}
The remaining assertions are clear from this inequality after an application of the 
Lebesgue Dominated Convergence Theorem. 
\end{proof}

\medskip

Denote by $\widehat G$ the set of (equivalence classes of)
unitary irreducible representations of $G$. For $ \pi\in\widehat G$,
write $H_{\pi}$ for the Hilbert space where
the representation $\pi$ acts.
Fixing an orthonormal basis $\{e_n\}_{n\in \bb{N}_\pi}$ of $H_\pi$
(where $\bb{N}_{\pi}$ is either finite or equals $\bb{N}$), 
we write $u^\pi_{i,j}(r)=\sca{\pi(r)e_j,e_i}$ for the coefficients of $\pi$.

Let $\pi\in\widehat G$ and $h\in T(G)$.
Define
\begin{align}
h_r(s,t) &= h(sr,tr), \ \ \ r,s,t\in G; \nonumber \\
h^\pi(s,t) &= \int_G h_r(s,t) \pi(r)dr\in \cl B(H_\pi); \nonumber \\
\tilde h^\pi(s,t) &=\pi(s)h^\pi(s,t)= \int_G h_r(s,t)\pi(sr)dr\in  \cl B(H_\pi), 
\nonumber \\
\intertext{where the integrals are understood in the weak sense. We also let}
h^\pi_{i,j}(s,t)& =\sca{h^\pi(s,t)e_j,e_i}= \int_G h_r(s,t) u^\pi_{i,j}(r)dr; \nonumber \\
\tilde h^\pi_{i,j}(s,t)& =\sca{\tilde h^\pi(s,t)e_j,e_i}= \int_G h_r(s,t) u^\pi_{i,j}(sr)dr .
\label{star}
\end{align}
If $\phi$ is a function on $G$, we denote by $\check\phi$
the function given by $\check\phi(s)=\phi(s^{-1})$, $s\in G$.

\begin{lemma}\label{tild}
Let $h\in T(G)$.
Then
$\tilde{h}^\pi_{i,j}\in \frak{S}(G)$ and $\nor{\tilde{h}^\pi_{i,j}}_{\frak{S}}\le\nor{h}_t$.
\end{lemma}
\begin{proof}
When $h=f\ot g$ is an elementary tensor, (\ref{star}) gives
\[
\tilde h_{i,j}^{\pi}(s,t)=\int_G f(x)u_{i,j}^\pi(x)g(ts\an x)dx
=\phi(st\an),
\]
where $\phi = (fu_{i,j}^\pi)*\check g\in A(G)$. Thus,
\[\tilde h_{i,j}^{\pi}(s,t) = \check\phi(ts\an)=(N\check\phi)(s,t)\]
and so, since $\check\phi\in A(G)\subseteq M^{\cb}A(G)$,
Theorem \ref{spro} shows that  $\tilde h_{i,j}^{\pi}$ is in $\frak S(G)$
and that
$\nor{\tilde h_{i,j}^{\pi}}_{\frak{S}}=\nor{\check\phi}_{M^{\cb}A(G)}.$
Now $A(G)$ embeds contractively in $M^{\cb}A(G)$ \cite{deCH} and so

\[
\nor{\check\phi}_{M^{\cb}A(G)}\le\nor{\check\phi}_{A(G)}
=\nor{\phi}_{A(G)}\le \nor{fu_{i,j}^{\pi}}_2\nor{g}_2.
\]
Thus,
\begin{align*}
\nor{\tilde h_{i,j}^\pi}_{\frak S} \leq \nor{fu_{i,j}^{\pi}}_2\nor{g}_2\leq
\nor{f}_2\nor{g}_2= \nor{h}_t.
\end{align*}
The same inequality holds for linear combinations  $\sum^N_{n=1} f_n\otimes g_n$,
and hence the linear operator
$\Phi$ given by $\Phi(h) = \tilde h_{i,j}^{\pi}$, defined on the
algebraic tensor product $L^2(G)\otimes L^2(G)$
extends to a bounded operator
$\Phi : T(G)\to\frak S(G)$;
clearly, $\nor{\Phi(h)}_{\frak{S}}\le\nor{h}_t$.

Now let $h=\sum_{i=1}^{\infty} f_i\ot g_i$ be an arbitrary
element of $T(G)$, and set
$h_n = \sum_{i=1}^n f_i\ot g_i$, $n\in \bb{N}$;
we show that $\Phi(h) = \tilde h_{i,j}^\pi$.
Since $h_n\to_{n\to \infty} h$ in $T(G)$, we have that
$\Phi(h_n)\to_{n\to \infty} \Phi(h)$ in $\frak{S}(G)$ and hence
$\Phi(h_n) \chi_{L\times L} \to \Phi(h)\chi_{L\times L}$ in $T(G)$
for every compact set $L\subseteq G$.
By \cite[Lemma 2.1]{st1}, a subsequence of
$(\Phi(h_n) \chi_{L\times L})_{n\in \bb{N}}$ converges m.a.e. to
$\Phi(h)\chi_{L\times L}$.

On the other hand, Proposition \ref{rem2} shows that
$\Phi(h_n) \to_{n\to \infty} \tilde h_{i,j}^\pi$ pointwise.
It follows that $\Phi(h) = \tilde h_{i,j}^\pi$.
\end{proof}

The proof of the following lemma,
which  follows readily from the definitions, is left to the reader.

\begin{lemma}\label{l_tr}
(i) \ If $h\in \frak{S}(G)$ then $h_r \in \frak{S}(G)$ and
$\|h_r\|_{\frak{S}} = \|h\|_{\frak{S}}$.

(ii) If $h\in T(G)$ then $h_r \in T(G)$ and $\|h_r\|_t \leq \Delta(r)^{-1} \|h\|_t$.
\end{lemma}

We do not know if $h^\pi_{i,j}$ always defines a Schur multiplier;  however, 
it suffices for our purposes to show that its restriction to a compact set does
define  a Schur multiplier; this is done in Lemma \ref{hschur}. In
Lemma  \ref{neol} we express this restriction in terms of $\tilde h^\pi_{k,j}$. 

\medskip

We thank the referee for the following remark.

\begin{remark}\label{remK}
If $h\in\frak S(G)$ is compactly supported, say $\supp h\subset K\times K$
where  $K\subseteq G$ is compact, then $h\in T(G)$.  
\end{remark}
\begin{proof} 
Indeed, given $\varepsilon>0$, writing 
$h=\sum_n\varphi_n\ot\psi_n$ with 
$$\nor{\sum_n  |\varphi_n|^2}_\infty\nor{\sum_n|\psi_n|^2}_\infty < (\|S_h\|+\varepsilon)^2,$$ 
we have 
\begin{align*}
\|h\|^2_t &\leq 
\left(\sum_n\|\varphi_n\|_2\|\psi_n\|_2\right)^2 
\le
\left(\sum_n \int_K|\varphi_n|^2\right)
\left(\sum_n \int_K|\psi_n|^2\right) \\ &= 
\left(\int_K\sum_n |\varphi_n|^2\right)
\left(\int_K\sum_n |\psi_n|^2\right) \\ 
&\leq m(K)^2\nor{\sum_n  |\varphi_n|^2}_\infty\nor{\sum_n|\psi_n|^2}_\infty
=m(K)^2(\|S_h\|+\varepsilon)^2 .
\end{align*}
\end{proof}

For the next lemma, recall (see for example \cite{blesmi} or \cite[Section 3]{spronk})
that $\frak S(G)$ can be identified with the
weak* Haagerup tensor product $L^\infty(G)\otimes^{w^*h} L^\infty(G)$
which coincides with the dual of the Haagerup tensor product
$L^1(G)$ $\otimes^h L^1(G)$, the duality being given by

\begin{equation}\label{eq_ind}
\du{w}{f\ot g} = \iint w(s,t)f(s)g(t)dsdt, \ \ w\in\frak S(G), \ f, g\in L^1(G).
\end{equation}

\begin{lemma}\label{hrcts}
Let $L\subseteq G$ be a compact set.
If $h\in\frak S(G)$ is supported in a compact set $K\times K$,
then the function  $r\to \du{(\chi_{L\times L}h_r)}{\omega}$
is continuous for every $\omega\in L^1(G)\ot^h L^1(G) $.
\end{lemma}
\begin{proof}
 For  $\omega=f\ot g$ where $f,g\in L^1(G)$, we have
\begin{align*}
&\du{(\chi_{L\times L}h_r)}{\omega} = \iint \chi_L(s)\chi_L(t) h(sr,tr)f(s)g(t)dsdt \\
& = \iint
 \chi_L(sr^{-1}) \chi_L(tr^{-1}) h(s,t) \Delta(r)^{-1}f(sr^{-1})\Delta(r)^{-1}g(tr^{-1})dsdt
\end{align*}
But, as $r\to e$, the function $s\to \chi_L(s r^{-1})f(s r^{-1})\Delta(r^{-1})$
tends to  $\chi_L f$  in the norm of $L^1(G)$; similarly for $\chi_L g$.
Therefore, since $h$ is bounded,
\[\du{(\chi_{L\times L}h_r)}{\omega}
\to_{r\to e} \iint \chi_L(s)h(s,t)\chi_L(t)f(s)g(t)dsdt.\]
It follows that $r\to \du{(\chi_{L\times L}h_r)}{\omega}$ is continuous
for every finite sum  $\omega=\sum f_n\ot g_n$. Since 
such elements $\omega$ form a
dense subset of $L^1(G)\otimes^h L^1(G)$, and
$\|h\|_\frak{S}=\|h_r\|_\frak{S}$ for all $r\in G$ (Lemma \ref{l_tr} (i)),
the conclusion follows.
\end{proof}

Suppose $h\in\frak S(G)$ is supported in a compact set $K\times K$
 and let $u:G\to\bb C$
be bounded and continuous.
For a compact set $L\subseteq G$, let
\begin{align}\label{int}
w_{h,L}(s,t) &= \chi_{L\times L}(s,t) \int_Gh(sr,tr)u(r)dr\\
& =
\chi_{L\times L}(s,t)\int_Gh(sr,tr)u(r)\chi_{L\an K}(r)dr \nonumber
\end{align}
(the second equality follows from the fact that
$sr\in K$ forces $r\in s\an K\subseteq L\an K$ if $s\in L$).
In the next lemma, we show that $w_{h,L}\in \frak{S}(G)$.
First note that, since $h\in\frak{S}(G)$, 
%\red{omit: [is assumed to satisfy $(K)$, and, in particular, belongs to $\frak{S}(G)$,]} 
the function
$r\to \du{(\chi_{L\times L}h_r)}{\omega}$ is bounded
and hence the integral
$\int \du{(\chi_{L\times L}h_r)}{\omega}u(r)\chi_{L\an K}(r)dr$
exists.

\begin{lemma} \label{hschur} If 
$h\in\frak S(G)$ is compactly supported, then
for every compact $L\subseteq G$
and every bounded continuous function $u:G\to\bb C$, the function
$w_{h,L}$ defined in (\ref{int})  is a Schur multiplier.
In particular, $\chi_{L\times L}h^\pi_{i,j}$ is a Schur multiplier.
\end{lemma} 
\begin{proof}
Suppose  $h$ is supported in a compact set $K\times K$.
By Lemmas \ref{l_tr}(i) and \ref{hrcts}, 
the linear mapping
\[
L^1(G)\otimes^h L^1(G)\to\bb C:
\omega \rightarrow \int \du{(\chi_{L\times L}h_r)}{\omega}u(r)\chi_{L\an K}(r)dr
\]
is bounded with norm not exceeding
$\nor{h}_\frak{S}\nor{u}_\infty m(L\an K)$; hence it
defines an element $v_{h,L} \in (L^1(G)\otimes^h L^1(G))^* = \frak{S}(G)$,
that is 
\begin{align} \label{666}
\du{v_{h,L}}{\omega} =   \int_G\du{(\chi_{L\times L}h_r)}{\omega}u(r)\chi_{L\an K}(r)dr ,
\end{align}
for all $\omega\in L^1(G)\otimes^h L^1(G)$.

We will show that
$v_{h,L} = w_{h,L}$ almost everywhere.

By (\ref{eq_ind}) and (\ref{666}),
if $\omega = f\ot g$ with $f,g\in L^1(G)$ then, 
applying Fubini's theorem (note that the integration with respect to $r$
is over a compact set), we obtain
\begin{align*}
& \iint\! \!v_{h,L}(s,t)f(s)g(t)dsdt = \du{v_L}{\omega}
=  \int\du{(\chi_{L\times L}h_r)}{\omega}u(r)\chi_{L\an K}(r)dr \\
&=\! \int\!\!\left(\iint\!(\chi_{L\times L}h_r)(s,t) f(s)g(t)dsdt\!\right) u(r)\chi_{L\an K}(r)dr\\
&=\!\iint\!\!\left(\int\!(\chi_{L\times L}h_r)(s,t) u(r)\chi_{L\an K}(r)dr \right)f(s)g(t)dsdt
\\ &=\!\iint\!w_{h,L}(s,t)f(s)g(t)dsdt.
\end{align*}

This shows that the function $v_{h,L} - w_{h,L}$
annihilates the algebraic tensor product
$L^1(G)\ot L^1(G)$, hence all of  $L^1(G\times G)$. It follows that it is zero as
an element of $L^\infty(G\times G)$.
\end{proof}

\begin{corollary}\label{c_wi}
Suppose $h\in\frak S(G)$ is compactly supported.
For every continuous bounded function $u:G\to\bb C$,
every  $T\in \cl B(L^2(G))$ and every $f,g\in L^2(G)$  supported in a compact set $L\subseteq G$ we have
\begin{align*}
(S_{w_{h,L}}(T)f,g) = \int_G u(r) (S_{h_r}(T)f,g) dr,
\end{align*}
where $w_{h,L}$ is as in (\ref{int}). In particular, \begin{align*}
(S_{h^{\pi}_{i,j}\chi_{L\times L}}(T)f,g) = \int_G u^{\pi}_{i,j}(r) (S_{h_r}(T)f,g) dr
 \end{align*}\label{l_peq}
for all $\pi\in \widehat G$ and all $i,j\in\bb N_\pi$.
\end{corollary}

\begin{proof}
Let $K\subseteq G$ be a compact set such that 
$\supp h\subseteq K\times K$.
Suppose first that  $T = T_k$ is a Hilbert-Schmidt operator
(here $k \in L^2(G\times G)$).
Using Fubini's theorem, we have
\begin{align*}
(S_{w_{h,L}}(T_k)f,g) & =
 \int_{G\times G} w_{h,L}(s,t) k(t,s)f(s) \overline{g(t)} dsdt \\
& =
 \int_{G\times G} \left(\int_{L\an K} h(sr,tr) u(r)dr\right)k(t,s) f(s)
\overline{g(t)}  dsdt\\
& =  \int_{L\an K} u(r)\left(\int_{G\times G} h(sr,tr) k(t,s)f(s)
\overline{g(t)}dsdt\right)  dr\\
& =  \int_{L\an K} u(r) (S_{h_r}(T_k)f,g) dr\, .
\end{align*}
\noindent If $T$ is arbitrary, let $(T_n)$ be a sequence of Hilbert-Schmidt operators 
(with operator norms uniformly bounded by $\|T\|$)
such that $T_n\rightarrow T$ in the weak* topology.
Then
\[
(S_{w_{h,L}}(T_n)f,g)\rightarrow (S_{w_{h,L}}(T)f,g)
\]
by the weak* continuity of $S_{w_{h,L}}$.
On the other hand, since
\[
|(S_{h_r}(T_n)f,g)|\leq \|T_n\|\|f\|_2\|g\|_2 \|h_r\|_{\frak{S}}
\leq \|T\|\|f\|_2\|g\|_2 \|h\|_{\frak{S}} 
\]
and $L^{-1}K$ has finite Haar measure,
the Lebesgue Dominated Convergence Theorem implies that
\[
\int_G u(r) (S_{h_r}(T_n)f,g) dr\rightarrow \int_G u(r) (S_{h_r}(T)f,g) dr.
\]
The conclusion follows.
\end{proof}

\begin{lemma}\label{l_claim} 
If $h\in T(G)$ then 
\begin{align}\label{cl}
\sum_{k\in\bb N_{\pi}}\nor{\tilde h^\pi_{k,j}}_{\frak{S}}^2\le \nor{h}_t^2.
\end{align}
\end{lemma}
\begin{proof}
Suppose first that $h$ is an elementary tensor, say $h = \phi\ot\psi$,
and recall that in this case
$ \nor{\tilde h^\pi_{k,j}}_{\frak{S}}\le \nor{u^\pi_{k,j}\phi}_2\nor{\psi}_2$
(see the proof of Lemma \ref{tild}). Now
\begin{eqnarray*}
\sum_k\nor{u^\pi_{i,k}\phi}^2
& = & \sum_k\int_G|u^\pi_{i,k}(s )\phi(s)|^2ds
=\int_G\sum_k|(\pi(s)e_k,e_i)|^2|\phi(s)|^2ds\\
& = & \int_G\sum_k|(e_k,\pi(s\an )e_i)|^2|\phi(s)|^2ds \\
& = & \int_G\nor{\pi(s\an )e_i}^2|\phi(s)|^2ds
=   \int_G|\phi(s)|^2ds =\nor{\phi}_2^2
\end{eqnarray*}
and so $\sum_k\nor{\tilde h^\pi_{k,j}}^2_{\frak{S}} \le \nor{h}_t^2$. \\
The same estimate persists when
$h$ is a finite sum $h=\sum_l\phi_l\ot\psi_l$: 
\begin{eqnarray*}
\sum_k\nor{\tilde h^\pi_{k,j}}^2_{\frak{S}}
& \leq & 
\sum_k\left(\sum_l\nor{u^\pi_{k,j}\phi_l}_2\nor{\psi_l}_2\right)^2 \le
\sum_l\sum_k\nor{u^\pi_{k,j}\phi_l}^2_2\sum_l\nor{\psi_l}^2_2\\
& = & \sum_l\nor{\phi_l}^2_2\sum_l\nor{\psi_l}^2_2.
\end{eqnarray*}
In particular,
$\sum\limits_{k=1}^N\|\tilde h^\pi_{k,j}\|_{\frak{S}}^2 \le  \|h\|_t^2$ 
for all $N\in\bb N_\pi$.
Since the map $h\to \tilde h^\pi_{k,j}:T(G)\to \frak S(G)$ is contractive
(see Lemma \ref{tild}), the last inequality holds for all $h\in T(G)$.
But $N$ is arbitrary, and so  (\ref{cl}) is proved for an
arbitrary $h$.
\end{proof}

We thank V. S. Shulman and L. Turowska for letting us include a proof of the  following
 lemma from an earlier version of \cite{stt}.

\begin{lemma}\label{l_unicom} 
For each $k$ and $\pi\in\hat G$, the functions
$\sum\limits_{l=m}^\infty |u_{k,l}^\pi|^2$ converge to zero uniformly
on compact sets, as $m\to\infty$.
\end{lemma}

\begin{proof} 
It suffices to consider the case where $H_\pi$ is 
infinite dimensional. Fix $k\in\bb N$ and let
\[
f_m(r)=\sum\limits_{l=m}^\infty|u^\pi_{k,l}(r)|^2
=\sum\limits_{l=m}^\infty|(\pi(r)e_k,e_l)|^2 =\nor{P_m\pi(r)e_k}^2,
\]
where $P_m$ is the projection on the closed subspace generated by $\{e_l:l\ge m\}$.

Since the function $r\to  \pi(r)e_k:G\to H_\pi$ is continuous, so is the function
$r\to  P_m\pi(r)e_k:G\to H_\pi$; thus each $f_m$ is a continuous function.

Since $(P_m)_{m\in \bb{N}}$ decreases to $0$, the sequence $(f_m(r))_{m\in\bb{N}}$ decreases to $0$ 
for each $r\in G$. By Dini's Theorem, the convergence is uniform on compact subsets of $G$.   
\end{proof}

\begin{lemma}\label{neol}
Assume $h\in T(G)$.
Let $\pi\in\widehat G$ and $i,j\in\bb N_\pi$. 
For each compact set $L\subseteq G$,
and all $f,g\in L^2(G)$, we have that
$$h^\pi_{i,j} (\chi_L f\otimes \chi_L g) =
\sum_k (\check u^\pi_{i,k}\ot\mathbf{1})\tilde h^\pi_{k,j} (\chi_L f\otimes \chi_L g)$$
and
$$\tilde h^\pi_{i,j} (\chi_L f\otimes \chi_L g) =
\sum_k (u^\pi_{i,k}\ot\mathbf{1}) h^\pi_{k,j} (\chi_L f\otimes \chi_L g)$$
in the norm of $T(G)$.
\end{lemma}

\begin{proof}
We prove the first formula; the second follows similarly.
We show first that the series
\begin{equation}\label{eq_hco}
\sum_k (\check u^\pi_{i,k}\ot\mathbf{1})\tilde h^\pi_{k,j} (\chi_L f\otimes \chi_L g)
\end{equation}
converges in the norm of $T(G)$.
Fix $\epsilon > 0$ and let $T\in \cl B(L^2(G))$ be a contraction.
For all $n < m$, 
\begin{align*}
\sum_{k=n}^m\nor{(\check u^\pi_{i,k}\chi_Lf)}^2
= &
\sum_{k=n}^m\int_L |\check u^\pi_{i,k}(r)|^2|f(r)|^2dr
= \int_L
\sum_{k=n}^m |\check u^\pi_{i,k}(r)|^2|f(r)|^2dr \\ \le &
\nor{f}_2^2 \sup_{r\in L\an}\sum_{k=n}^m |u^\pi_{i,k}(r)|^2 .
\end{align*}
Using Lemma \ref{l_unicom},  we can choose $n < m$ so that
\begin{equation}\label{eqnew}
\sum_{k=n}^m\nor{(\check u^\pi_{i,k}\chi_Lf)}^2 < \epsilon^2 \|h\|_t^{-2}
\|\chi_L \overline{g}\|^{-2}.
\end{equation}
By the Cauchy-Schwarz inequality,
\begin{align*}
& \left|\left\langle T, \sum_{k=n}^m (\check u^\pi_{i,k}\ot\mathbf{1})\tilde h^\pi_{k,j} (\chi_L f\otimes \chi_L g)\right\rangle\right|^2 \\
& \leq
\left(\sum_{k=n}^m \left|\sca{S_{\tilde h^\pi_{k,j}}(T)(\check u^\pi_{i,k} \chi_L f),\chi_L \overline{g}}\right|\right)^2 
\leq
\left(\sum_{k=n}^m \left\| S_{\tilde h^\pi_{k,j}}(T)(\check u^\pi_{i,k} \chi_L f)\right\|
\left\|\chi_L \overline{g}\right\|\right)^2\\
&\le 
\sum_{k=n}^m\nor{S_{\tilde h^\pi_{k,j}}(T)}^2
\sum_{k=n}^m\nor{(\check u^\pi_{i,k} \chi_L f)}^2 \left\|\chi_L \overline{g}\right\|^2  < \epsilon ,
\end{align*}
where for the last inequality we have used (\ref{eqnew}) and Lemma \ref{l_claim}.
It follows that the series (\ref{eq_hco}) converges in the norm of $T(G)$;
let $\Lambda$ be its sum.
By \cite[Lemma 2.1]{st1}, there exists a sequence of partial sums of
(\ref{eq_hco}) that converges marginally almost everywhere to $\Lambda$.

On the other hand, for every $s,t\in G$, we have
\begin{align*}
h^\pi_{i,j}(s,t) &= (h^\pi(s,t)e_j,e_i) = (\tilde h^\pi(s,t)e_j,\pi(s)e_i) \\
&
=\sum_k(\tilde h^\pi(s,t)e_j,e_k)(e_k,\pi(s)e_i)= \sum_k u^\pi_{i,k}(s\an)\tilde h^\pi_{k,j}(s,t) \, ,
\end{align*}
and hence the series (\ref{eq_hco}) converges pointwise to
the function $h^\pi_{i,j} (\chi_L f\otimes \chi_L g)$. It follows that
$\Lambda = h^\pi_{i,j} (\chi_L f\otimes \chi_L g)$ and the proof is complete.
\end{proof}

\begin{lemma}\label{tgsg}
Let $\cl S\subseteq T(G)$ be a norm closed $L^\infty (G)$-bimodule.
Each $h\in\cl S$ is the norm limit of a sequence
$(h_n)$ with $h_n=h\chi_{K_n\times L_n}\in\cl S\cap\frak S(G)$
where $K_n$ and $L_n$ are compact sets.  
\end{lemma}
\begin{proof}
Let $ h=\sum_{i=1}^{\infty} f_i\ot g_i$,
where  $\sum_{i=1}^{\infty} \nor{f_i}_2^2$ and $\sum_{i=1}^{\infty} \nor{g_i}_2^2$
are finite.
Given $n\in \bb N$, let
\begin{align*}
A_n =\left\{s\in G: \sum_{i=1}^{\infty} |f_i(s)|^2\le n\right\}, \qquad
B_n =\left\{t\in G: \sum_{i=1}^{\infty} |g_i(t)|^2\le n\right\}
\end{align*}
and choose compact sets $K_n\subseteq A_n$ and   $L_n\subseteq B_n$
such that $m(A_n\setminus K_n)<2^{-n}$ and  $m(B_n\setminus L_n)<2^{-n}$.
Setting $h_n=h\chi_{K_n\times L_n}$, we see that
\\
(a) $h_n \in\cl S$ because  $\cl S$ is an $L^{\infty}(G)$-bimodule, and
\\
(b) $h_n\in\frak S(G)$ because
$h_n(s,t) =  \sum_{i=1}^{\infty} (\chi_{K_n}f_i)(s)(\chi_{L_n}g_i)(t)$ ($s,t\in G$), where
$\sum_{i=1}^{\infty}| (\chi_{K_n}f_i)(s)|^2\le n$ and
$\sum_{i=1}^{\infty}|(\chi_{L_n}g_i)(t)|^2\le n$  a.e..

It is straightforward to see that \;$\nor{h-h_n}_t\to 0$.
\end{proof}

\begin{lemma}\label{48}
Let $h\in T(G)$ be supported in a compact set $K\times K$.
Then $h$ belongs to the $T(G)$-closed
linear span of
\[
\{\chi_{K\times K}h^\pi_{i,j}:  \pi\in\widehat G, i,j\in\bb N_\pi\}. 
\]
\end{lemma}

\begin{proof}
Suppose $T\in B(L^2(G))$ satisfies
\[\du{T}{\chi_{K\times K}h^\pi_{i,j}}_t=0\quad\text{for all }\;  \pi\in\widehat G, \, i,j\in\bb N_\pi.
\]
We will show that $\du{T}{h}_t=0$.

Recall that
$h^\pi_{i,j}(s,t)= \int_Gh(sr,tr)u^\pi_{i,j}(r)dr$.
We may write $\chi_{K\times K}h^\pi_{i,j}$ in the form
\[\chi_{K\times K}h^\pi_{i,j}
=\chi_{K\times K}(\chi_{K\an K}u^\pi_{i,j}\Delta\an)\star h
\]
where $\star$ is the action of $L^1(G)$ on $T(G)$ given by
\[(g\star h)(s,t):=\int_G h(sr,tr)g(r)\Delta(r)dr \]
which satisfies $\nor{g\star h}_{T(G)}\le \nor{g}_{L^1(G)}\nor{h}_{T(G)}$.
Thus the hypothesis gives
$$\du{S_{\chi_{K\times K}}(T)}{(\chi_{K\an K}u^\pi_{i,j}\Delta\an)\star h}_t
= \du{T}{\chi_{K\times K}(\chi_{K\an K}u^\pi_{i,j}\Delta\an)\star h}_t = 0.$$
Now let $f$ be  a continuous function supported in the
compact set  $K\an K$. Then $f\Delta$ is continuous and vanishes outside
$K\an K$; hence it is the limit, uniformly in $K\an K$, of a
sequence $(g_n)$ of linear combinations of coefficients  $u^\pi_{i,j}$
of irreducible representations $\pi$ of $G$ (see
\cite[Theorem 3.27, 3.31 and Proposition 3.33]{folland} or
\cite[13.6.5 and 13.6.4]{dixmier}).
Hence $g_n\Delta\an\to f$ uniformly in  $K\an K$ (observe that $\Delta\an$
is  continuous, hence bounded, on compact sets).
Each $g_n$ is a linear combination of coefficients $u^\pi_{i,j}$ and therefore
$\du{S_{\chi_{K\times K}}(T)}{(\chi_{K\an K}g_n\Delta\an)\star h}_t=0$ for each $n$.
Since  $f=f\chi_{K\an K}$, it follows that
$\du{S_{\chi_{K\times K}}(T)}{f\star h}_t=0.$
Now let $\{f_\alpha\}$ be an approximate identity for $L^1(G)$
consisting of non-negative continuous functions with $\nor{f_\alpha}_1=1$,
all supported in $K\an K$.
Then a standard argument shows that $\nor{f_\alpha\star h-h}_t\to 0$ and we obtain
\[
\du{T}{h}_t=\du{T}{\chi_{K\times K}h}_t
=\du{S_{\chi_{K\times K}}(T)}{h}_t
=\lim_\alpha\du{S_{\chi_{K\times K}}(T)}{f_\alpha\star h}_t
=0
\]
which proves the lemma.
\end{proof}

\bigskip

We now proceed to the proof of Theorem \ref{th_satlcg}.
We first show that
\begin{equation}\label{eq_ine}
\Sat(J) \subseteq (\Bim (J^{\perp}))_{\perp}.
\end{equation}
Let $u\in J$, $h\in T(G)$, $w\in \frak{S}(G)$ and
$T\in J^{\perp}\subseteq \vn(G)$.
Then, if $\tau\in A(G)^*$ satisfies
$P^*(\tau) = T$, using (\ref{pn}), we have
\begin{eqnarray*}
\langle  S_w(T),N(u)h\rangle_t & = & \langle  T,  N(u)wh\rangle_t =
\langle  P^*(\tau), N(u)wh\rangle_t\\
& = & \langle  \tau,  P(N(u)wh)\rangle_a = \langle \tau, uP(wh)\rangle_a = 0
\end{eqnarray*}
since $u\in J$ and $P(wh) \in A(G)$, hence $uP(wh)\in J$. 
Thus, $S_w(T)$ annihilates $\Sat(J)$ by Proposition \ref{newsat}.
Since $\{S_w(T): T\in J^\perp, w\in\frak{S}(G)\}$ 
generates $\Bim(J^\perp)$, (\ref{eq_ine}) is established.

\medskip

For the reverse inclusion, suppose that $h\in (\Bim (J^{\perp}))_{\perp}$.
By Lemma \ref{tgsg}, we may
assume that there exists a compact set $K\subseteq G$
such that $\supp h\subseteq K\times K$ and $h \in \frak{S}(G)\cap (\Bim (J^{\perp}))_{\perp}$.

The steps of the argument are the following:

\medskip
\noindent\textbf{Step 1. } {\em  If $T \in J^\perp$ then
$S_{h_r}(T) = 0$ for every $r\in G$.}

\smallskip

\noindent {\em Proof. } A direct verification using relation (\ref{eq_p}) shows that
if $r\in G$ then $\Delta(r)^{-1} P(h) = P(h_r)$.
By Lemma \ref{l_tr} (i), $h_r\in \frak{S}(G)$.
It suffices to prove that $\sca{S_{h_r}(T)\xi,\eta}=0$ whenever
$\xi$ and $\eta$ are in $L^\infty(G)\cap L^2(G)$.
In this case,
$w:=\xi\ot\bar\eta$ is in $T(G)\cap\frak S(G)$ and,
if $\tau\in A(G)^*$ is such that $T_\tau=T$, then
\begin{align*}
\sca{S_{h_r}(T)\xi,\eta} & =
\langle S_{h_r}(T), w\rangle_t = \langle T, h_r w\rangle_t =
\langle \tau, P(h_r w)\rangle_a \\ & =
\Delta(r)^{-1} \langle \tau, P(h w_{r^{-1}})\rangle_a
= \Delta(r)^{-1} \langle T, h w_{r^{-1}}\rangle_t \\
& = \Delta(r)^{-1} \langle S_{w_{r^{-1}}}(T), h\rangle_t = 0,
\end{align*}
since $h$ annihilates $\Bim(J^{\perp})$. Hence, $S_{h_r}(T) = 0$ for
every $r\in G$.

\medskip

\noindent{\bf Step 2. } {\em If $T \in J^\perp$ then
$S_{\chi_{L\times L}h^{\pi}_{i,j}}(T) = 0$
for all $\pi\in \widehat G$, all $i,j$,
and all compact sets $L\subseteq G$.}

\smallskip

\noindent This follows from Step 1 and Corollary \ref{l_peq}.

\medskip

\noindent{\bf Step 3. }  {\em If $T \in J^\perp$ then
$S_{\tilde h^{\pi}_{i,j}}(T) = 0$.}

\smallskip

\noindent {\em Proof. } Step 2 and  Lemma \ref{neol} imply that,
for every $f,g\in L^2(G)$ and every compact set $L\subseteq G$, we have
\[
(S_{\tilde{h}^{\pi}_{i,j}}(T)\chi_Lf,\chi_Lg) = \sum_{k=1}^{\infty}
(S_{\chi_{L\times L}h^{\pi}_{k,j}}(T)u^\pi_{i,k}\chi_Lf,\chi_Lg) =0.
\]
This implies that $S_{\tilde{h}^{\pi}_{i,j}}(T) = 0$.

\medskip

\noindent{\bf Step 4. } {\em If $\pi\in \widehat G$ and $L$ is a compact subset
of  $G$, then 
$\tilde{h}^{\pi}_{i,j}\chi_{L \times L}\in \Sat(J)$ for all $i,j$.}

\smallskip

\noindent {\em Proof. } Since 
$\tilde{h}^{\pi}_{i,j}\in \frak{S}(G)$ (Lemma \ref{tild}) and
$\tilde{h}^{\pi}_{i,j}(sr,tr) = \tilde{h}^{\pi}_{i,j}(s,t)$, Theorem \ref{spro}
implies that $\tilde{h}^{\pi}_{i,j} =N(v)$,
for some $v\in M^{\cb}A(G)$. We claim that  $vA(G) \subseteq J$.
Indeed, for every $\tau\in J^{\perp}$ and every $w\in T(G)$, we have, by Step 3,
\begin{align*} 
0 &= \langle S_{\tilde{h}^{\pi}_{i,j}}(T_{\tau}), w\rangle_t
= \langle S_{N(v)}(T_\tau), w\rangle_t
= \langle T_\tau, N(v)w\rangle_t \\
&= \langle \tau, P(N(v)w)\rangle_a = \langle \tau, vP(w)\rangle_a  
\end{align*}    
where we have used relation (\ref{pn}). 
Note that $vP(w)\in A(G)$ since $P(w)\in A(G)$ and $v\in M^{\cb}A(G)$.
This equality shows, by the Hahn-Banach theorem, that $vP(w)\in J$.
Thus, since $P$ is surjective, $vA(G) \subseteq J$.
We may choose $w\in T(G)$ such that
$u:=P(w)$ satisfies $u|_{LL\an}=1$. Then
$N(u)\chi_{L\times L}=\chi_{L\times L}$ and so
\[
\tilde{h}^{\pi}_{i,j} \chi_{L\times L}=
N(v)\chi_{L\times L}=N(v)N(u)\chi_{L\times L}=N(vu)\chi_{L\times L}.
\]
But $vu=vP(w)\in J$. Thus,
$\tilde{h}^{\pi}_{i,j}\chi_{L\times L}\in \Sat(J)$.

\medskip

\noindent{\bf Step 5. }
{\em If $\pi\in \widehat G$ and $L$ is a compact subset of  $G$,
then $h^{\pi}_{i,j}\chi_{L \times L}\in \Sat(J)$ for all $i,j$.}

\smallskip

\noindent {\em Proof. } This is a direct consequence of Lemma \ref{neol}.

\medskip

\noindent{\bf Step 6. } $h\in\Sat(J)$.

\smallskip

\noindent {\em Proof. }
By Lemma \ref{48}, $h$ is in the $T(G)$-norm
closed linear span of elements of the form
$h^{\pi}_{i,j}\chi_{L\times L}$, so
$h\in\Sat(J)$.

\medskip

The proof of Theorem \ref{th_satlcg} is complete.

\section{Jointly invariant subspaces}\label{s_invsub}

In this section, we characterise the common weak* closed invariant
subspaces of Schur multipliers and a class of completely bounded maps
arising from a canonical representation of the measure algebra of $G$. 

Let $\rho : G\rightarrow \cl B(L^2(G))$, $r\rightarrow \rho_r$, be the
right regular representation of $G$ on $L^2(G)$, that is, the representation
given by $(\rho_r f)(s) = \Delta(r)^{1/2} f(sr)$, $s,r\in G$, $f\in L^2(G)$.

Let $M(G)$ be the Banach algebra of all bounded complex Borel measures on $G$.
Following \cite{nrs} (see also \cite{smith_spronk}), we define a representation
$\Gamma$ of $M(G)$ on $\cl B(L^2(G))$ by letting
$$\Gamma(\mu) (T) = \int_{G} \rho_rT\rho_r^*d\mu(r), \ \ \ T\in \cl B(L^2(G)),$$
the integral being understood in the weak sense
(that is, for every $h\in T(G)$ and every $T\in \cl B(L^2(G))$
the formula
$\langle\Gamma(\mu) (T),h\rangle_t
= \int_{G} \langle\rho_rT\rho_r^*,h\rangle_t d\mu(r)$ holds).
This representation was studied by E. St\o{}rmer \cite{stor}, F. Ghahramani \cite{gh},
M. Neufang \cite{n} and M. Neufang, Zh.-J. Ruan and N. Spronk  \cite{nrs},
among others.

Denote, as is customary, by ${\rm Ad}\rho_r$ the map on $\cl B(L^2(G))$ given by
${\rm Ad}\rho_r(T)$ $= \rho_r T \rho_r^*$, $T\in \cl B(L^2(G))$; since 
${\rm Ad}\rho_r$ is a (bounded) weak* continuous map, it has a (bounded) predual
$\theta_r : T(G)\to T(G)$.

\begin{lemma}\label{l_pred}
Let $r\in G$. Then $\theta_r(h) = \Delta(r^{-1})h_{r^{-1}}$, for every $h\in T(G)$.
\end{lemma}
\begin{proof}
By linearity and continuity
(see Lemma \ref{l_tr} (ii)),
it suffices to check the formula when $h$ is an elementary tensor,
say $h = f\otimes \bar g$ for some $f,g\in L^2(G)$.
For $T\in \cl B(L^2(G))$, we have

\begin{align*}
\du{T}{\theta_r(h)}_t & =
\langle {\rm Ad }\rho_r(T),h\rangle_t
= \langle \rho_rT\rho_{r\an},f\otimes \bar g\rangle_t
= (T(\rho_{r\an} f),\rho_{r\an} g) \\
&= \langle T,(\rho_{r\an} f)\otimes \overline{\rho_{r\an} g}\rangle_t.
\end{align*}
However, if $s,t\in G$ then
\begin{align*}
(\rho_{r\an} f\otimes \overline{\rho_{r\an}g})(s,t)
& =  (\rho_{r\an} f)(s) (\overline{\rho_{r\an} g})(t) =
\Delta(r\an)f(sr\an)\bar g(tr\an)\\
& =  \Delta(r\an)h_{r\an}(s,t).
\end{align*}
The proof is complete.
\end{proof}

\begin{lemma}\label{l_exid}
Let $V\subseteq T(G)$ be a norm closed $L^{\infty}(G)$-bimodule such that
$\theta_r(V)\subseteq V$ for each $r\in G$. Then there exists a closed ideal 
$J\subseteq A(G)$ such that $V = \Sat(J)$.
\end{lemma}

\begin{proof} 
Let
\[J = \{u\in A(G) : N(u) \chi_{L\times L}  \in V 
\text{ for every compact set } L\subseteq G\}.\]
Since $A(G)$ embeds contractively into $M^{cb}A(G)$ and the map $N$ 
is continuous, it  is clear that $J$ is a closed subspace of $A(G)$.
We check that $J$ is an ideal: if $u\in J$, $v\in A(G)$
and $L\subseteq G$ is a compact subset, then
\[
N(uv)\chi_{L\times L} = (Nu)(Nv)\chi_{L\times L} \in \frak{S}(G)V\subseteq V
\]
since $N(v)\in\frak S(G)$ by Theorem  \ref{spro} and $V$, 
being a closed $L^{\infty}(G)$-bimodule, is invariant under $\frak{S}(G)$.

Clearly   $\Sat(J)\subseteq V$. To show that $V\subseteq\Sat(J)$, let $h\in V$.
By Lemma \ref{tgsg}, we may assume that $\supp h\subseteq K\times K$ 
and $h \in \frak{S}(G)\cap V$ for some compact set $K\subseteq G$.
In order to  conclude that $h\in \Sat(J)$ it suffices, by Lemma \ref{48}, to prove that
$h^{\pi}_{i,j}\chi_{L\times L} \in \Sat(J)$
for every irreducible representation $\pi$ of $G$,
every $i,j\in \bb{N}_\pi$ and every compact set $L\subseteq G$.

The function $r\to u^{\pi}_{i,j}(r) h_r$, $G\to T(G)$, is continuous (Lemma \ref{l_tr}(ii)) 
and hence the integral
$$\omega := \int_{L^{-1}K} u^{\pi}_{i,j}(r) h_r dr =
\int_{L^{-1}K} u^{\pi}_{i,j}(r) \Delta(r) \theta_r(h) dr$$
exists as a Bochner integral and defines  an element of
$T(G)$. The second equality shows that $\omega$ is in the closed 
linear span of $\{\theta_r(h):r\in G\}$.
But $V$ is invariant under $\theta_r$, and hence 
$\omega\in V$.

We claim that $ \chi_{L\times L}\omega = \chi_{L\times L} h^{\pi}_{i,j}$. To see this,
let $T = T_k$ be a Hilbert-Schmidt operator of the form
$k = f\otimes g$ with $f,g\in L^2(G)$. Then
\begin{align*}
\du{ \chi_{L\times L}\omega}{T}_t & =
\int_{L^{-1}K} u^{\pi}_{i,j}(r) \du{ \chi_{L\times L}h_r}{T}_t dr\\
& =  \int_{L^{-1}K} \iint_{G\times G} u^{\pi}_{i,j}(r) \chi_{L\times L}(s,t) h_r(s,t) f(t) g(s) dsdt dr \\
&= \du{\chi_{L\times L} h^{\pi}_{i,j}}{T}_t 
\end{align*}
(the last equality follows as in the proof of Corollary \ref{c_wi}).
This proves the claim.

Thus 
$\chi_{L\times L} h^{\pi}_{i,j}$ is in $V$.
Since $V$ is a norm closed $L^\infty(G)$ bimodule, using  Lemma \ref{neol} we conclude that
$\chi_{L\times L}\tilde{h}^{\pi}_{i,j} \in V$.
By Theorem \ref{spro}, since $\tilde{h}^{\pi}_{i,j}\in\frak S(G)$,
there exists
$v \in M^{\cb}A(G)$ such that $\tilde{h}^{\pi}_{i,j} = N(v)$.

We claim that $vA(G)\subseteq J$. Indeed, for every  $u\in A(G)$, if
$L\subseteq G$ is a compact set, 
$$\chi_{L\times L}N(vu) = (\chi_{L\times L}\tilde{h}^{\pi}_{i,j})N(u)\in V\frak{S}(G)\subseteq V$$
and thus $vu\in J$ by the definition of $J$.

 Since $P$ is surjective,  we may choose $w\in T(G)$ such that
$u:=P(w)$ satisfies $u|_{LL\an}=1$. Then
$N(u)\chi_{L\times L}=\chi_{L\times L}$ and so
\[
\tilde{h}^{\pi}_{i,j} \chi_{L\times L}=
N(v)\chi_{L\times L}=N(v)N(u)\chi_{L\times L}=N(vu)\chi_{L\times L}.
\]
 Since $vA(G)\subseteq J$, we obtain  $vu=vP(w)\in J$. Thus, 
$\tilde{h}^{\pi}_{i,j}\chi_{L\times L}\in \Sat(J)$.

Using Lemma \ref{neol} again, we obtain $\chi_{L\times L} h^{\pi}_{i,j}\in \Sat(J)$
 and the proof is complete.
\end{proof}

\begin{theorem}\label{th_chim}
Let $\cl U\subseteq \cl B(L^2(G))$ be a weak* closed subspace.
The following are equivalent:

(i) \ \ the space $\cl U$ is invariant under the mappings $S_w$ and $\Gamma(\mu)$,
for all $w\in \frak{S}(G)$ and all $\mu\in M(G)$;

(ii) \ the space $\cl U$ is invariant under the mappings $S_w$
and ${\rm Ad}\rho_r$, for all
$w\in \frak{S}(G)$ and all $r\in G$;

(iii) there exists a closed ideal $J\subseteq A(G)$ such that $\cl U = \Bim(J^{\perp})$.
\end{theorem}
\begin{proof}
(i)$\Rightarrow$(ii) This follows by choosing $\mu$
to be the point mass at  $r\in G$. 

\noindent (ii)$\Rightarrow$(iii) \ Let $V = \cl U_{\perp}$;
then $V$ is a norm closed subspace of $T(G)$,
invariant under the maps of the form $m_w$ ($w\in \frak{S}(G)$) and $\theta_r$ ($r\in G$).
By Lemma \ref{l_exid}, there exists a closed ideal $J\subseteq A(G)$ such that
$V = \Sat(J)$. Hence  $\cl U = \Bim(J^{\perp})$ by Theorem \ref{th_satlcg}.

(iii)$\Rightarrow$(i) 
Let $T\in \cl U$ and $\mu\in M(G)$. To show that $\Gamma(\mu)(T)\in\cl U$
 it suffices, by Theorem \ref{th_satlcg}, to show
that $\langle \Gamma(\mu)(T),w\rangle_t = 0$ for every $w\in \Sat(J)$.
But, if $w\in \Sat(J)$ then, for all $r\in G$,
$$\langle \rho_r T \rho_r^*,w\rangle_t = \langle T,\theta_r(w)\rangle_t = 0$$
since $\Sat(J)$ is clearly $\theta_r$-invariant.    
It follows that
$$\langle \Gamma(\mu)(T),w\rangle_t
= \int_G \langle \rho_r T \rho_r^*,w\rangle_t d\mu(r) = 0 \, .$$
Since $\cl U = \Bim(J^{\perp})$ is invariant under all Schur multipliers,
the proof is complete.
\end{proof}

\begin{corollary}\label{c_lat}
Let $\cl L$ be the set of
all weak* closed subspaces of  $\cl B(L^2(G))$
which are invariant under both $\frak{S}(G)$ and $\Gamma(M(G))$. Then
$$\cl L= \{\Bim(J^{\perp}) : J\subseteq A(G) \mbox{ a closed ideal}\}$$
and $\cl L$ is a lattice under the operations of intersection
and closed linear span $\vee$. Moreover, 
$$\Bim(J_1^{\perp})\cap \Bim(J_2^{\perp}) = \Bim((J_1 + J_2)^{\perp}),$$
$$\Bim(J_1^{\perp}) \vee \Bim(J_2^{\perp}) = \Bim((J_1\cap  J_2)^{\perp}).$$
\end{corollary} 
\begin{proof}
The description of $\cl L$ is contained in Theorem \ref{th_chim}.
The first identity will be proved in Proposition \ref{p_inf}.
The inclusion
$\Bim(J_1^{\perp}) \vee \Bim(J_2^{\perp}) \subseteq \Bim((J_1\cap  J_2)^{\perp})$
is trivial. The reverse
inclusion follows directly from the
definition of $\Bim$ and the fact that
$(J_1\cap J_2)^{\perp} = \overline{J_1^{\perp} + J_2^{\perp}}^{w^*}$.
\end{proof}

\begin{lemma}\label{le_appr}
Assume that  $A(G)$ has a (possibly unbounded)
approximate identity. Then
$J^{\bot }=\Bim (J^{\bot })\cap \vn (G)$
for every ideal $J$ of $A(G)$.
\end{lemma}
\begin{proof}
Let $T \in \Bim (J^{\bot })\cap \vn (G)$. Since $T$ is in $\vn (G)$, it is of the form
 $T=P^*(\tau)$ for some $\tau\in A(G)^*$ (see relation (\ref{P})).
By Theorem \ref{th_satlcg},
$T\in (\Sat(J))^{\bot}$. 
By Proposition \ref{newsat},
for all $v\in  J$ and all $h\in T(G)$, 
\[
0=\du{T}{(N(v)h}_t = \du{P^*(\tau)}{N(v)h}_t =
\du{\tau}{P(N(v)h)}_a = \du{\tau}{vP(h)}_a
\]
(using relation (\ref{pn})).
Since $A(G)$ has an approximate identity  and
the map $P : T(G) \to A(G)$ is surjective,
there is a net $(h_i)$ in $T(G)$ such that
$vP(h_i)\to v$ in $A(G)$.
It follows that
$\du{\tau}{v}_a=0$ for all $v\in J$
and therefore $T\in J^{\bot }$ as claimed.
\end{proof}

\begin{remark} {\rm It follows from Theorem \ref{th_chim} that
$\Bim$ maps the set of all weak* closed $A(G)$-invariant 
subspaces in $\vn(G)$ onto
the set of all weak* closed masa-bimodules in $\cl B(L^2(G))$
which are invariant under
conjugation by $\rho_r$, $r\in G$.

Using Theorem \ref{th_satlcg} we see that $\Sat$
maps the set of all closed ideals of $A(G)$ onto
the set of all closed $L^{\infty}(G)$-bimodules in $T(G)$ which are invariant under
$h\to h_r$. }
\end{remark}

Combining this with Lemma \ref{le_appr} gives the following:

\begin{corollary}
If $A(G)$ has an approximate identity,
the maps $\Bim$ and $\Sat$ are bijective.
\end{corollary}

\noindent Note that the class of groups for which $A(G)$ 
possesses an approximate identity contains, but is
strictly larger than, the class of all amenable groups
(see \cite[Remark 4.5]{lt} for a relevant discussion).
It is unknown whether this class contains all locally compact groups;
it does contain those groups having the `approximation 
property' of Haagerup and Kraus \cite{hakr}. It is now known that there 
are groups failing the approximation property \cite{laf, hala}.

\begin{question}
Does the conclusion of  Lemma \ref{le_appr} hold for an arbitrary second countable
locally compact group?
\end{question}

\noindent {\bf Remark }
Theorem \ref{th_chim} describes the class of all weak* closed subspaces of
 $\cl B(L^2(G))$ which are invariant under
 $\Gamma(M(G))$ and under all Schur multipliers. If instead we consider
only invariant Schur multipliers, namely $N(M^{\cb}A(G))$, we obtain a strictly
larger class; consider, for example,  $\vn(G)$.

\section{The extremal bimodules}\label{s_eb}

In this section, we relate Theorem \ref{th_satlcg} to the
extremal masa-bimodules associated with a subset of $G\times G$
\lq\lq of Toeplitz type''.
We start by recalling some notions and results from \cite{a} and \cite{eks}
in the special case that we will use.

A subset $E$ of $G\times G$ is called {\it $\omega$-open}
if it is marginally equivalent to the union of a countable set of Borel rectangles. 
The complements of $\omega$-open sets are called {\it $\omega$-closed}.
If $F \subseteq G\times G$ is an $\omega$-closed set, an operator $T\in \cl B(L^2(G))$
is said to be \emph{supported by} $F$ if
$$(A \times B)\cap F \simeq \emptyset \ \Longrightarrow \  P(B)TP(A) = 0,$$
for all measurable rectangles $A\times B\subseteq G\times G$, where $P(A)$
denotes  the  orthogonal projection from $L^2(G)$ onto $L^2(A)$.
Given a masa-bimodule $\cl U$, there exists a
smallest, up to marginal equivalence, $\omega$-closed subset
$F \subseteq G\times G$
such that every operator in $\cl U$ is supported by $F$;
we call $F$ the \emph{support} of $\cl U$.
Given an $\omega$-closed set $F \subseteq G\times G$,
there exists \cite{a} a largest weak* closed masa-bimodule
$\frak{M}_{\max}(F)$ and a smallest weak* closed masa-bimodule
$\frak{M}_{\min}(F)$ with support $F$. The masa-bimodule
$\frak{M}_{\max}(F)$ is the space
of all $T\in \cl B(L^2(G))$ supported on $F$.
We say that an $\omega$-closed set $F \subseteq X\times Y$ satisfies
{\em operator synthesis} if $\frak{M}_{\max}(F)=\frak{M}_{\min}(F)$.

Let $F$ be an $\omega$-closed subset of $G\times G$; define
$$\Phi(F) = \{\psi \in T(G) : \psi\chi_{F} = 0 \mbox{ m.a.e.}\}$$
and
$$\Psi(F) = \overline{\{\psi \in T(G) : \psi \mbox{ vanishes on an $\omega$-open neighbourhood of }F\}}^{\|\cdot\|_t}.$$
It was shown in \cite{st1} that $\Phi(F)^{\perp} = \frak{M}_{\min}(F)$ and
$\Psi(F)^{\perp} = \frak{M}_{\max}(F)$.
For an $L^{\infty}(G)$-bimodule $V$ in $T(G)$, we let
$\nul (V)$ be the largest, up to marginal equivalence, 
$\omega$-closed subset $F$ of $G\times G$
such that $h|_F = 0$ for all $h\in V$ \cite{st1}.
Then a closed
$L^{\infty}(G)$-bimodule $V\subseteq T(G)$
satisfies $\Psi(F)\subseteq V \subseteq \Phi(F)$ if and only if $\nul (V) = F$.

\smallskip

For a closed set  $E\subseteq G$, let
\begin{align*}
I(E)& =\{f\in A(G) : f(x)=0, x\in E\} \quad\text{and }\;  J(E)=\overline{J_0(E)}, \\
\text{where }\;\;
J_0(E)& =\{f\in A(G) : f \mbox{ vanishes on a neighbourhood of } E\}.
\end{align*}
If $J\subseteq A(G)$ is a closed ideal, denote by $Z(J)$ the set of common zeroes
of functions in $J$: 
\[Z(J) = \{s\in G : f(s) = 0 \mbox{ for all } f\in J\}.\]
Then
$J(E)\subseteq J\subseteq I(E)$ if and only if $Z(J) = E$.
If $J(E) = I(E)$ then one says that $E$ satisfies {\em spectral synthesis}.

For a subset $E\subseteq G$, we set
\[
E^* = \{(s,t) : ts^{-1}\in E\}.
\]

The relation between the notions of a null set and a zero set is described
in the next proposition. 

\begin{proposition}\label{p_nulls}
Let $J\subseteq A(G)$ be a closed ideal. Then $\nul (\Sat(J)) = (Z(J))^*$.
\end{proposition}
\begin{proof}
Let $E = Z(J)$.
By the definition of $\Sat(J)$, using \cite[Lemma 2.1]{st1} 
every element $h$ of $\Sat(J)$ is a m.a.e. limit
of a  sequence of finite sums of elements of the form
$\phi_iN(u_i)\chi_{L_i\times L_i}$, where $\phi_i \in \mathfrak{S}(G)$,
$u_i \in J$ and $L_i$ is a compact subset of $G$.   But
$ N(u)\chi_{L\times L}$ vanishes on $E^*$ for every $u \in J$
and every compact subset $L$ of $G$. Thus $h$ vanishes m.a.e.  on $E^*$ and
it follows that $E^*\subseteq \nul \Sat(J)$.

Conversely, suppose that $(s,t)\not\in E^*$, that is, $ts^{-1}\not\in E$.
Let $U$ be a compact neighbourhood of $ts^{-1}$ disjoint from $E$
and let $v\in A(G)$ be a function vanishing on an open neighbourhood of $E$
such that $v|_{U} = 1$ \cite[Lemma 3.2]{eym}. Then $v\in J$ and, if $K, L\subseteq G$
are compact neighbourhoods of $s$ and $t$ such that
$LK^{-1}\subseteq U$, then $(Nv)\chi_{U\times U}$
takes the value $1$ on $K\times L$.
Thus every  $(s,t)\not\in E^*$ has a relatively compact open neighbourhood
$W_{(s,t)}\subseteq K\times L$
disjoint from $\nul \Sat(J)$ up to a marginally null set.
Taking a countable subcover of the cover
$\{W_{(s,t)}\}$ of $(E^*)^c$   we obtain an $\omega$-open neighbourhood
of $(E^*)^c$ disjoint from $\nul \Sat(J)$ up to a marginally null set.
It follows that $\nul \Sat(J)\subseteq E^*$ up to a marginally null set
and the proof is complete.
\end{proof}

\begin{corollary}\label{c_im}
Let $V\subseteq T(G)$ be a norm closed
$L^{\infty}(G)$-bimodule such that
if $h\in V$ then $h_r\in V$, for every $r\in G$. 
Then there exists a closed subset $E\subseteq G$
such that $\nul V = E^*$.
\end{corollary}
\begin{proof}
By  Lemma \ref{l_exid}, there exists an ideal $J\subseteq A(G)$ such that
$V = \Sat(J)$. Let $E = Z(J)$; by Proposition \ref{p_nulls}, $\nul V = E^*$.
\end{proof}

\begin{theorem}\label{th_mma}
Let $E\subseteq G$ be a closed set. The following hold:

(i) \ $\Bim(I(E)^{\perp}) = \frak{M}_{\min}(E^*)$;

(ii) $\Bim(J(E)^{\perp}) = \frak{M}_{\max}(E^*)$.
\end{theorem}
\begin{proof}
(i) By Theorem \ref{th_satlcg}, $\Bim(I(E)^{\perp})_{\perp} = \Sat({I(E)})$.
By Proposition \ref{p_nulls}, $\nul \Sat({I(E)}) = E^*$ and hence
$\frak{M}_{\min}(E^*) \subseteq \Bim(I(E)^{\perp})$ by the minimality of $\frak{M}_{\min}(E^*)$.

To prove the reverse inclusion,  
let $T = T_{\tau}\in I(E)^{\perp}$ and $w\in \frak{S}(G)$.
If $h\in T(G)$ vanishes on $E^*$ then $wh$ vanishes on $E^*$.
Relation (\ref{eq_p}) now shows that $P(wh)\in I(E)$ and so 
$$\langle S_w(T),h\rangle_t = \langle T,wh\rangle_t
= \langle \tau, P(wh)\rangle_a = 0,$$
showing that $S_w(T)\in \Phi(E^*)^\perp=\frak{M}_{\min}(E^*)$. 
Thus, $\Bim(I(E)^{\perp}) \subseteq \frak{M}_{\min}(E^*)$ and the proof is complete.

(ii)                
Observe that for each $v\in J_0(E)$ and each compact set $L\subseteq G$,
the element $N(v)\chi_{L\times L}$ is in $\Psi(E^*)$. 
By continuity of the map $N$, the same holds for $v\in J(E)$.
It follows from the definition of
$\Sat(J(E))$ that $\Sat(J(E))\subseteq \Psi(E^*)$.

On the other hand, by Proposition \ref{p_nulls}, $\nul \Sat({J(E)}) = E^*$ and,
since $\Sat(J(E))$ is a closed $L^{\infty}(G)$-bimodule,
the minimality property of $\Psi(E^*)$ shows that
$\Psi(E^*)\subseteq \Sat(J(E))$.

Hence $\Sat({J(E)}) = \Psi(E^*)$. By \cite{st1} and Theorem \ref{th_satlcg},
 $\Bim(J(E)^{\perp}) = \frak{M}_{\max}(E^*)$.
\end{proof}

It is worthwhile to isolate the following characterisation of reflexive jointly 
invariant subspaces, which is is an immediate consequence of Theorem \ref{th_chim}.

\begin{theorem}\label{refl}Let $\cl U\subseteq \cl B(L^2(G))$ be a 
reflexive subspace.
Then  $\cl U$ is invariant under all mappings $S_w, w\in \frak{S}(G)$
and ${\rm Ad}\rho_r, r\in G$, 
if and only if 
there exists a closed set $E\subseteq G$ such that $\cl U = 
\frak{M}_{\max}(E^*)$.
\end{theorem}

As a corollary to Theorem \ref{th_mma}, we obtain the following result of \cite{lt}:

\begin{theorem}[\cite{lt}]\label{th_opsynth}
Assume that $A(G)$ has an approximate identity. Then a closed set $E\subseteq G$
satisfies spectral synthesis if and only if the set $E^*\subseteq G\times G$
satisfies operator synthesis.
\end{theorem}
\begin{proof}
Assume $E$ satisfies spectral synthesis.
Then $I(E)=J(E)$ and it follows from  Theorem \ref{th_mma} that
$\frak{M}_{\min}(E^*)= \frak{M}_{\max}(E^*)$.
Conversely, if $\frak{M}_{\min}(E^*)= \frak{M}_{\max}(E^*)$ then,
by Theorem  \ref{th_mma},
 $\Bim (I(E)^{\bot }) = \Bim (J(E)^{\bot })$
and now, by Lemma \ref{le_appr}, $I(E)=J(E)$.
\end{proof}

\begin{question}
Let $E\subseteq G$ be a closed subset.
 Is every weak* closed masa-bimodule $\cl U$
with $\frak{M}_{\min}(E^*)\subseteq \cl U \subseteq \frak{M}_{\max}(E^*)$
of the form $\cl U = \Bim(J^{\perp})$ for some closed ideal $J\subseteq A(G)$? 
\end{question}

 In view of Theorem \ref{th_chim},
the above question asks, in other words, whether
there exist closed sets $E\subseteq G$ such that $E^*$ supports
a weak* closed masa-bimodule
not invariant under conjugation by the unitaries $\rho_s$, $s\in G$
(see Section \ref{s_invsub}).
If such a set exists, it will necessarily be non-synthetic; for if
$E$ is synthetic, then Theorem \ref{th_opsynth}
gives  $\frak{M}_{\min}(E^*)=\frak{M}_{\max}(E^*)$
and consequently no such bimodule exists.

\section{Relative synthesis}\label{s_rs}

In this section, we obtain an extension of Theorem \ref{th_opsynth}
which links relative spectral synthesis to relative operator synthesis.
Theorem \ref{th_rs} was proved by
K. Parthasarathy and R. Prakash in \cite[Theorem 4.6]{pp} 
under the assumption that $\cl X$ is an  $A(G)$-invariant subspace of $\vn(G)$ 
and $G$ is compact. In our  result we
assume that $\cl X$ is  weak* closed  and 
$A(G)$ possesses an approximate identity.

We recall the relevant definitions. Let
$\cl X\subseteq \vn(G)$ be an $A(G)$-invariant subspace. 
A closed subset $E\subseteq G$
is called \emph{$\cl X$-spectral} \cite{kl} if
$\cl X\cap J(E)^\perp = \cl X\cap I(E)^\perp$. 
This notion has a natural operator theoretic version:
if $\cl U$ is a weak* closed masa-bimodule and $F$ is an
$\omega$-closed set, we say that
$F$ is \emph{$\cl U$-operator synthetic} if
$\cl U\cap \frak{M}_{\min}(F) = \cl U\cap \frak{M}_{\max}(F)$.
The latter notion was
defined in \cite{pp} for subsets of $G\times G$, where $G$ is a
compact group.

\begin{proposition}\label{p_inf}
(i) \ Let $\cl X_1$ and $\cl X_2$ be weak* closed invariant subspaces of $\vn(G)$.
Then $\Bim(\cl X_1\cap \cl X_2) = \Bim(\cl X_1)\cap \Bim(\cl X_2)$.

(ii) Let $J_1$ and $J_2$ be 
closed ideals of $A(G)$. Then
$\Sat(J_1)  \cap \Sat(J_2) = \Sat(J_1\cap J_2).$
\end{proposition}
\begin{proof}
(i) Let $J_i \subseteq A(G)$ be a closed ideal with $J_i^{\perp} = \cl X_i$, $i = 1,2$.
Let $J = \overline{J_1 + J_2}$; then $J$ is
the smallest closed ideal of $A(G)$ containing both $J_1$ and $J_2$.
Note that 
\begin{equation}\label{eq_sat}
\Sat(J_1)^{\perp} \cap \Sat(J_2)^{\perp}\subseteq \Sat(J)^{\perp}.
\end{equation}
Indeed, if $T\in \cl B(L^2(G))$ annihilates both $\Sat(J_1)$ and $\Sat(J_2)$
then, by Proposition \ref{newsat}, $T$ annihilates $N(J_1) T(G)$ and $N(J_2)T(G)$,
hence  their sum; continuity of the map $N$ shows  that $T$ must annihilate
 $N(J)T(G)$ and hence  $\Sat(J)$.

It is obvious that
$\Bim(\cl X_1\cap \cl X_2) \subseteq \Bim(\cl X_1)\cap \Bim(\cl X_2)$.
Suppose that $T\in \Bim(\cl X_1)\cap \Bim(\cl X_2)$. By Theorem \ref{th_satlcg}
and (\ref{eq_sat}),
$T\in \Bim(J^{\perp})$.
But $\Bim(J^{\perp})\subseteq \Bim(J_1^\perp\cap J_2^\perp)
=\Bim(\cl X_1\cap \cl X_2)$, since $J_i\subseteq J$, $i = 1,2$.
Thus, $T\in \Bim(\cl X_1\cap \cl X_2)$ and the proof is complete. 

(ii) It follows from Corollary \ref{c_lat} and Theorem \ref{th_satlcg} that
\[\overline{\Sat(J_1)^{\perp} + \Sat(J_2)^{\perp}}^{w^*}= \Sat(J_1\cap J_2)^{\perp}.\] 
Taking pre-annihilators, the result follows.
\end{proof}

\begin{theorem}\label{th_rs}
Assume that $A(G)$ possesses an approximate identity. 
Let $E\subseteq G$ be a closed set,
$\cl X\subseteq \vn(G)$  a weak* closed invariant subspace 
and $\cl U = \Bim(\cl X)$.
The following are equivalent:

(i) \ $E$ is $\cl X$-spectral;

(ii) $E^*$ is $\cl U$-operator synthetic.
\end{theorem}
\begin{proof}
Suppose $E$ is $\cl X$-spectral. Then $\cl X\cap J(E) ^\perp= \cl X\cap I(E)^\perp$.
By Proposition \ref{p_inf} and Theorem \ref{th_mma},
$\cl U\cap \frak{M}_{\min}(E^*) = \cl U\cap \frak{M}_{\max}(E^*)$.
Thus, $E^*$ is $\cl U$-operator synthetic.

Conversely, suppose $E^*$ is $\cl U$-operator synthetic.
Then 
$$\cl U\cap \frak{M}_{\min}(E^*) \cap \vn(G)
= \cl U\cap \frak{M}_{\max}(E^*) \cap \vn(G).$$
By Lemma \ref{le_appr} and   Theorem \ref{th_mma},
$\cl X\cap J(E) = \cl X\cap I(E)$, that is, $E$ is $\cl X$-spectral.
\end{proof}

\noindent
{\bf Acknowledgement. } We are very grateful to the referee for carefully
reading our paper and suggesting a number of improvements.

\end{document}